\documentclass[12pt]{amsart}

\usepackage{amsmath}
\usepackage{amssymb}
\usepackage{amsfonts}
\usepackage{amsthm}
\usepackage{enumerate}
\usepackage{hyperref}
\usepackage{color}

\theoremstyle{plain}
\newtheorem{thm}{Theorem}[section]

\newtheorem{prop}[thm]{Proposition}
\newtheorem{lem}[thm]{Lemma}
\newtheorem{cor}[thm]{Corollary}

\newtheorem{conj}[thm]{Conjecture}

\theoremstyle{definition}

\newtheorem{exmp}[thm]{Example}

\newtheorem{rem}[thm]{Remark}

\newtheorem{dfns-rems}[thm]{Definitions and Remarks}
\newtheorem{notas-rems}[thm]{Notations and Remarks}
\newtheorem{exmps-rems}[thm]{Examples and Remarks}

\DeclareMathOperator{\Ass}{Ass}
\DeclareMathOperator{\height}{height}
\DeclareMathOperator{\depth}{depth}
\DeclareMathOperator{\Ext}{Ext}

\begin{document}


\title[weighted edge ideal]{Cohen-Macaulay edge-weighted edge ideals of very well-covered graphs}


\author[Seyed Amin Seyed Fakhari]{Seyed Amin Seyed Fakhari}

\address{Seyed Amin Seyed Fakhari, School of Mathematics, Statistics and Computer Science,
College of Science, University of Tehran, Tehran, Iran.}

\email{aminfakhari@ut.ac.ir}

\author[Kosuke Shibata]{Kosuke Shibata}

\address{Kosuke Shibata, Department of Mathematics, Faculty of Science, Okayama University, Kita-ku Okayama 700--8530, Japan.}

\email{pfel97d6@okayama-u.ac.jp}

\author[Naoki Terai]{Naoki Terai}

\address{Naoki Terai, Department of Mathematics, Faculty of Science, Okayama University, Kita-ku Okayama 700--8530, Japan.}

\email{terai@okayama-u.ac.jp}

\author[Siamak Yassemi]{Siamak Yassemi}

\address{Siamak Yassemi, School of Mathematics, Statistics and Computer Science,
College of Science, University of Tehran, Tehran, Iran.}

\email{yassemi@ut.ac.ir}


\begin{abstract}
We characterize unmixed and Cohen-Macaulay edge-weighted edge ideals of very well-covered graphs. We also provide examples of oriented graphs which have unmixed and non-Cohen-Macaulay vertex-weighted edge ideals, while the edge ideal of their underlying graph is Cohen-Macaulay. This disproves a conjecture posed by Pitones, Reyes and Toledo.
\end{abstract}






\thanks{}


\maketitle


\section{Introduction} \label{sec1}

In this article, a graph means a simple graph without loops,
multiple edges, and isolated vertices. Let $G$ be a graph with the vertex set $V(G) = \{x_1,
\dots , x_n\}$ and with the edge set $E(G)$.
Suppose $w : E(G)\longrightarrow \mathbb{Z}_{>0}$ is an edge weight on $G$.
We write $G_w$ for the pair $(G,w)$ and call it an edge-weighted graph.
Let $S=\mathbb{K}[x_1, \ldots, x_n]$ be the polynomial ring in $n$ variables over a field $\mathbb{K}$.
The \textit{(edge-weighted) edge ideal} of an edge-weighted graph $G_w$ was introduced in \cite{PaSW} and it is defined as$$I(G_w)=\big((x_ix_j)^{w(x_ix_j)} \, |\,  x_ix_j\in E(G)\big),$$(by abusing the notation, we identify the edges of $G$ with quadratic squarefree monomials of $S$). Paulsen and Sather-Wagstaff \cite{PaSW} studied the primary decomposition of these ideals. They also investigated unmixedness and Cohen-Macaulayness of these ideals, in the case that $G$ is a cycle, tree or a complete graph. The aim of this paper is to continue this study. In Section \ref{sec3}, we characterize unmixed and Cohen-Macaulay properties of the edge-weighted edge ideals of very well-covered graphs (see Section \ref{sec2} for the definition of very well-covered graphs). Our results can be seen as generalizations of the results concerning the Cohen-Macaulay property of usual edge ideals of very well-covered  graphs (see e.g., \cite{CHHKTT, CRT, CV, HH}). For other aspects of ring-theoretic study for very well-covered graphs, see e.g., \cite{BM, KTY, MMCRTY, SF}.

Another kind of generalization of edge ideals is considered in \cite{HLMRV, PRT, PRV}. Indeed, Pitones, Reyes and Toledo \cite{PRT} introduced the \textit{vertex-weighted} edge ideal of an oriented graph as follows. Let $\mathcal{D}=(V(\mathcal{D}),E(\mathcal{D}))$ be an oriented graph with $V(\mathcal{D})=\{x_1,\ldots , x_n\}$, and let $w : V(\mathcal{D})\longrightarrow \mathbb{Z}_{>0}$ be a \textit{vertex-weighted} on $\mathcal{D}$. Set $w_j:=w(x_j)$. The \textit{vertex-weighted edge ideal} of $\mathcal{D}$ is defined as
$$I(\mathcal{D})=\big(x_ix_j^{\omega_j} \, |\,  (x_i,x_j\in E(\mathcal{D})\big).$$Pitones, Reyes and Toledo proposed the following conjecture.

\begin{conj} \label{ConjHLMRV}\cite[Conjecture 53]{PRT}
Let $\mathcal{D}$ be a vertex-weighted oriented graph and let $G$ be its underlying graph. If $I(\mathcal{D})$ is unmixed and $S/I(G)$ is Cohen-Macaulay, then $S/I(\mathcal{D})$ is Cohen-Macaulay.
\end{conj}

In Section \ref{sec4}, we provide counterexamples for this conjecture.

We close this introduction by mentioning that unmixed and Cohen-Macaulay properties of vertex-weighted edge ideals of vertex-weighted oriented very well-covered  graphs are studied by Pitones, Reyes and Villarreal \cite{PRV}.


\section{Preliminaries} \label{sec2}

In this section, we provide the definitions and basic facts which will be used in the next sections. We refer to \cite{Die} and \cite{Vi1} for detailed information.

Let $G$ be a graph with the vertex set $V(G) = \{x_1\dots, x_n\}$ and with the edge set $E(G)$. For every integer $1\leq i\leq n$, the {\it degree} of $x_i$, denoted by $\deg_Gx_i$ , is the number of edges of $G$ which are incident to $x_i$. For $F \subset E(G)$ we denote $(V(G), E(G) \setminus F)$ by $G-F$. For a family $F$ of 2-element subsets of $V(G)$ the graph $(V(G), E(G) \cup F)$ is denoted by $G+F$. A subset $C \subset V(G)$ is a \textit{vertex cover} of $G$ if
every edge of $G$ is incident with at least one vertex in $C$. A vertex cover $C$ of $G$ is called \textit{minimal} if there
is no proper subset of $C$ which is a vertex cover of $G$. A subset $A$ of $V(G)$ is called an \textit{independent set} of $G$
if no two vertices of $A$ are adjacent. An independent set $A$ of $G$ is \textit{maximal} if there exists no independent set
which properly includes $A$. Observe that $C$ is a minimal vertex cover of $G$ if and only if $V(G)\setminus C$
is a maximal independent set of $G$. A subset $M\subseteq E(G)$ is a {\it matching} if $e\cap e'=\emptyset$, for every pair of edges $e, e'\in M$. If every vertex of $G$ is incident to an edge in $M$, then $M$ is a {\it perfect matching} of $G$. A graph $G$ without isolated vertices is said to be {\it very well-covered} if $|V(G)|$ is an even integer and every maximal independent subset of $G$ has cardinality $|V(G)|/2$.

A graph $G$ is called \textit{Cohen-Macaulay} if $S/I(G)$ is a Cohen-Macaulay
ring. Here, $I(G)$ is the {\it edge ideal} of $G$, which is defined as$$I(G)=\big(x_ix_j \, |\,  x_ix_j\in E(G)\big).$$
An ideal $I\subset S$ is \textit{unmixed} if the associated primes of $S/I$ have the same height. It is well known that $I$ is unmixed if $S/I$ is a Cohen-Macaulay ring. A graph $G$ is called \textit{unmixed} if the minimal vertex covers of $G$ have the same size. It can be easy seen that $G$ is an unmixed graph
if and only if $I(G)$ is an unmixed ideal. Also, note that $\height I(G)$ is equal to the smallest size of vertex covers of $G$.

We introduce polarization according to \cite{SV}. Let $I$ be a monomial ideal of
$S=\mathbb{K}[x_1,\ldots,x_n]$ with minimal generators $u_1,\ldots,u_m$,
where $u_j=\prod_{i=1}^{n}x_i^{a_{i,j}}$, $1\leq j\leq m$. For every $i$
with $1\leq i\leq n$, let $a_i=\max\{a_{i,j}\mid 1\leq j\leq m\}$, and
suppose that $$T=\mathbb{K}[x_{11},x_{12},\ldots,x_{1a_1},x_{21},
x_{22},\ldots,x_{2a_2},\ldots,x_{n1},x_{n2},\ldots,x_{na_n}]$$ is a
polynomial ring over the field $\mathbb{K}$. Let $I^{{\rm pol}}$ be the squarefree
monomial ideal of $T$ with minimal generators $u_1^{{\rm pol}},\ldots,u_m^{{\rm pol}}$, where
$u_j^{{\rm pol}}=\prod_{i=1}^{n}\prod_{k=1}^{a_{i,j}}x_{ik}$, $1\leq j\leq m$. The
monomial $u_j^{{\rm pol}}$ is called the {\it polarization} of $u_j$, and the ideal $I^{{\rm pol}}$
is called the {\it polarization} of $I$. It is well known that polarization preserves the height of ideal. Moreover, $I$ is an unmixed ideal if and only if $I^{{\rm pol}}$ is an unmixed ideal.

Finally, we recall the concept of Serre's condition.
Let $I$ be a monomial ideal of $S$.
For a positive integer $k$, the ring $S/I$ satisfies the Serre's condition $(S_k)$
if$$\depth (S/I)_{\frak{p}} \ge \min\{\dim (S/I)_{\frak{p}},\, k \}$$for every $\frak{p} \in {\rm Spec}(S/I)$.

\begin{lem} \label{Serre}\cite[Lemma 3.2.1]{Sch}
The following two conditions are equivalent.
\begin{enumerate}
\item $S/I$ satisfies the Serre's condition $(S_k)$.
\item For every integer $i$ with $0\leq i <\dim S/I$, the inequality$$\dim \Ext_S ^{n-i}(S/I, S) \le i-k$$holds, where the dimension of zero module is defined to be $-\infty$.
\end{enumerate}
\end{lem}


\section{Edge-weighted edge ideal of very well-covered graphs} \label{sec3}

In this section, we study the unmixed and Cohen-Macaulay properties of edge-weighted edge ideal of very well-covered graphs. We first recall some known facts about the structure of very well-covered graphs.

\begin{lem} \label{matching}\cite{GV2}
Let $G$ be a very well-covered graph. Then $G$ has a perfect matching.
\end{lem}

By the above lemma, we may assume that the vertices of the very well-covered graph $G$ are labeled such that the following condition is satisfied.

\medskip

(*)
$V(G)=X \cup Y$, $X \cap Y=\emptyset$,
where $X=\{x_1,\ldots,x_h\}$ is a minimal vertex cover of $G$
and $Y=\{y_1,\ldots, y_h\}$ is a maximal independent set of $G$
such that  $\{x_1 y_1, \ldots, x_hy_h\} \subset E(G)$.

Following the notations of condition (*), for the rest of this section, we set $S=K[x_1,\ldots,x_h, y_1,\ldots, y_h]$. For later use, we recall the following  characterization of very well-covered graphs.

\begin{prop}{\cite{CRT, MoReVi}}\label{th:unmix_ijk}
Let $G$ be a graph with $2h$ vertices, which are not isolated. Assume that the vertices of $G$ are labeled such that the condition (*) is satisfied. Then $G$ is
very well-covered if and only if the following hold.
\begin{enumerate}
\item[(i)] If $z_i x_j$, $y_j x_k\in E(G)$, then $z_i x_k\in E(G)$
for distinct indices $i$, $j$ and $k$ and for $z_i\in \{ x_i, y_i\}$.
\item[(ii)] If $x_iy_j \in E(G)$, then $x_ix_j \notin E(G)$.
\end{enumerate}
\end{prop}

We are now ready to state and prove the first main result of this paper, which characterizes edge-weighted very well-covered graphs with unmixed edge ideals.

\begin{thm} \label{th:unmix}
Let $G$ be a very well-covered graph with $2h$ vertices and let $w$ be an edge weight on $G$. Moreover, assume that the vertices of $G$ are labeled in such a way that the condition (*) is satisfied. Then $I(G_{w})$ is
unmixed if and only if the following hold.
\begin{enumerate}
\item[(i)] If $x_iz_j \in E(G)$, then $w(x_iz_j)\le w(x_iy_i)$ and  $w(x_iz_j) \le w(x_jy_j)$
for distinct indices $i,j$, and for any vertex $z_j\in \{x_j, y_j\}$.
\item[(ii)] If $z_ix_j$ and $y_j x_k$ are edges of $G$, then $w(z_ix_k)\le w(z_ix_j)$ and  $w(z_ix_k)\le w(y_jx_k)$
for distinct indices $i,j,k$ and for $z_i\in \{x_i, y_i\}$,  or
for distinct indices $j, i=k$ and for $z_i=y_i$.
\end{enumerate}
\end{thm}

\begin{proof}
Set $J:=I(G_w)^{\rm pol}$.

Suppose $I(G_w)$ is unmixed. Then $J$ is an unmixed ideal of height $h$. In particular, for every integer $i$ with $1\leq i\leq h$, any minimal prime of $J$ contains exactly one variable whose first index is $i$. We first prove condition (i). Assume that $x_iz_j \in E(G)$. Set $a:=w(x_iz_j)$ and $b:=w(x_i y_i)$. As $$x_{i1}x_{i2}\cdots x_{ia}z_{j1}z_{j2}\cdots z_{ja} \in J,$$there is a minimal prime $\mathfrak{p}_1$ of $J$ with $x_{ia}\in \mathfrak{p}_1$. By contradiction, suppose $a>b$. It follows from$$x_{i1}x_{i2}\cdots x_{ib}y_{i1}y_{i2}\cdots y_{ib}\in J$$that at least one of the variable $x_{i1}, x_{i2}, \dots, x_{ib},y_{i1}, y_{i2}, \dots, y_{ib}$ belongs to $\mathfrak{p}_1$. Therefore, $\mathfrak{p}_1$ contains two variables with first index $i$, which is a contradiction. Hence, $a\leq b$.

Now, set $c:=w(x_jy_j)$ and suppose $a>c$. As$$x_{i1}x_{i2}\cdots x_{ia}z_{j1}z_{j2}\cdots z_{ja}\in J,$$there is a minimal prime $\mathfrak{p}_2$ of $J$ with $z_{ja}\in \mathfrak{p}_2$. Also, it follows from$$x_{j1}x_{j2}\cdots x_{jc}y_{j1}y_{j2}\cdots y_{jc}\in J$$that at least one of the variable $x_{j1}, x_{j2}, \dots, x_{jc},y_{j1}, y_{j2}, \dots, y_{jc}$ belongs to $\mathfrak{p}_2$. Therefore, $\mathfrak{p}_2$ contains two variables with first index $j$, which is a contradiction. Hence, $a\leq c$.

Next, we prove condition (ii). Assume that $z_ix_j$ and $y_jx_k\in E(G)$. Since $G$ is unmixed, it follows from Proposition \ref{th:unmix_ijk} that $z_ix_k \in E(G)$ (this is trivially true, if $i=k$ and for $z_i=y_i$). Set $d:=w(z_ix_k)$, $e:=w(z_ix_j)$, $f:=w(y_jx_k)$. Suppose $d>e$. Since $w(z_ix_k)=d$, it follows that$$z_{i1}z_{i2}\cdots z_{i(d-1)}x_{k1}x_{k2}\cdots x_{kf} \notin J.$$Thus, there is a minimal prime $\mathfrak{p}_3$ of $J$ with$$z_{i1}z_{i2}\cdots z_{i(d-1)}x_{k1}x_{k2}\cdots x_{kf}\notin \mathfrak{p}_3.$$Hence, neither of the variables $z_{i1}, z_{i2}, \dots, z_{i(d-1)}, x_{k1}, x_{k2}, \dots, x_{kf}$ belongs to $\mathfrak{p}_3$. Then we deduce from$$z_{i1}z_{i2}\cdots z_{ie}x_{j1}x_{j2}\cdots x_{je}, y_{j1}y_{j2}\cdots y_{jf}x_{k1}x_{k2}\cdots x_{kf}\in J$$that $x_{js}, y_{jt}\in \mathfrak{p}_3$, for some positive integers $s$ and $t$. This is a contradiction, as no minimal prime of $J$ can contain both of $x_{js}$ and $y_{jt}$. Thus, $d\le e$.

Suppose $d>f$. Since$$z_{i1}z_{i2}\cdots z_{ie}x_{k1}x_{k2}\cdots x_{k(d-1)}\notin J,$$there is a minimal prime $\mathfrak{p}_4$ of $J$ which contains neither of the variables$$z_{i1}, z_{i2}, \dots, z_{ie}, x_{k1}, x_{k2}, \dots, x_{k(d-1)}.$$It follows from $$z_{i1}z_{i2}\cdots z_{ie}x_{j1}x_{j2}\cdots x_{je}, y_{j1}y_{j2}\cdots y_{jf}x_{k1}x_{k2}\cdots x_{kf}\in J$$that $x_{j\ell}, y_{jr}\in \mathfrak{p}_4$, for some positive integers $\ell$ and $r$. This is again a contradiction. Therefore, $d\le f$.

\vspace{.1in}

We now prove the reverse implication. Suppose conditions (i) and (ii) hold and assume by contradiction that $I(G_w)$ is not unmixed. Hence, $J$ is not an unmixed ideal. Thus, there is a minimal prime $\mathfrak{p}$ of $J$ such that $x_{jp}, y_{jq}\in \mathfrak{p}$, for some integers $j, p, q\geq 1$. As above set $c:=w(x_j y_j)$.

Assume $p>c$. Since $\mathfrak{p}$ is a minimal prime of $J$, there is $i \neq j$ and $z_i\in \{x_i, y_i\}$ such that $z_ix_j \in E(G)$ and $w(z_ix_j) \geq p$. Then by (i), we have $c \ge w(z_ix_j)\ge p>c$, which is a contradiction. Hence  $p\le c$.

Suppose $q>c$. Since $\mathfrak{p}$ is a minimal prime of $J$, there is $k \neq j$ such that $y_jx_k \in E(G)$ with $w(y_jx_k)\ge q$. Then by (i) we have $c \ge w(y_jx_k)\ge q >c$, which is a contradiction. Therefore, $q\le c$.

Since $\mathfrak{p}$ is a minimal prime of $J$, there is $\ell \neq j$ and $z_{\ell}\in \{x_{\ell}, y_{\ell}\}$ such that $z_{\ell}x_j \in E(G)$ with $\alpha:=w(z_{\ell}x_j) \ge p$ and$$z_{\ell1}, z_{\ell2}, \dots, z_{\ell\alpha} \not\in \mathfrak{p}.$$Similarly,, there is $r \ne j$ such that $y_jx_r \in E(G)$ with $\beta:=w(y_jx_r)\ge q$ and$$x_{r1}, x_{r2}, \dots, x_{r\beta} \not\in \mathfrak{p}.$$By Proposition \ref{th:unmix_ijk}, $z_{\ell}x_r\in E(G)$. Set $\gamma:=w(z_{\ell}x_r)$. It follows from condition (ii) that $\gamma\le \alpha$ and $\gamma\le \beta$. Thus,$$z_{\ell1}, z_{\ell2}, \dots, z_{\ell\gamma}, x_{r1}, x_{r2}, \dots, x_{r\gamma}\notin \mathfrak{p}.$$This contradicts$$z_{\ell1}z_{\ell2}\cdots z_{\ell\gamma}x_{r1}x_{r2}\cdots x_{r\gamma} \in J.$$Hence, $I(G_w)$ is an unmixed ideal.
\end{proof}

\begin{rem}
Let $G$ be a very well-covered graph and let $w$ be an edge weight, such that $I(G_w)$ is an unmixed ideal. Assume that the vertices of $G$ are labeled in such a way that the condition (*) is satisfied. It follows from Theorem \ref{th:unmix} that if $x_iy_j, x_jy_i\in E(G)$, then $w(x_iy_i)=w(x_jy_j)=w(x_iy_j) =w(x_jy_i)$.
\end{rem}

Our next goal is to provide a combinatorial characterization for Cohen-Macaulayness of edge-weighted edge ideal of very well-covered graphs. First we summarize the known results concerning the Cohen-Macaulay property of a (non-weighted) very well-covered graph.

\begin{lem}\label{lem:order} \cite{CRT} Let $G$ be an unmixed graph with $2h$ vertices,
which are not isolated, and assume that the vertices of $G$ are labeled such the condition (*) is satisfied.
If $G$ is a Cohen-Macaulay graph then there exists a suitable simultaneous change
of labeling on both $\{x_i\}_{i=1}^h$ and $\{y_i\}_{i=1}^h$
(i.e., we relabel  $(x_{i_1}, \ldots, x_{i_h})$ and $(y_{i_1}, \ldots,
y_{i_h})$ as  $(x_1, \ldots, x_h)$ and $(y_1, \ldots, y_h)$ at
the same time), such that $x_iy_j \in E(G)$ implies $i \leq j$.
\end{lem}

Hence, for a Cohen-Macaulay very well-covered graph $G$ satisfying the condition (*), we may assume
that

\vspace{.2in} (**) $x_iy_j \in E(G)$ implies $i \leq j$.
\vspace{.2in}

Now we recall Cohen-Macaulay criterion for very well-covered graphs. See also \cite{CHHKTT, CV} for different characterizations.

\begin{thm}\label{Crit:CM}\cite{CRT} Let $G$ be a graph with
$2h$ vertices, which are not isolated and assume that the vertices of $G$ are labeled such that the conditions (*) and (**) are satisfied.
Then the following conditions are equivalent:
\begin{enumerate}
\item $G$ is Cohen-Macaulay;
\item $G$ is unmixed;
\item The following conditions hold:
\begin{enumerate}
\item[(i)] If $z_ix_j, y_jx_k \in E(G)$, then $z_ix_k \in E(G)$ for distinct indices $i, j, k$ and
for $z_i \in \{ x_i, y_i \}$;
\item[(ii)] If $x_iy_j \in E(G)$, then $x_ix_j \notin E(G)$.
\end{enumerate}
\end{enumerate}
\end{thm}

In order to study the Cohen-Macaulay property of edge-weighted edge ideal of very well-covered graphs, we introduce an operator which allows us to construct a new weighted very well-covered graph from a given one.

Let $G_w$ be a weighted very well-covered graph with $n=2h$ vertices and assume that the vertices of $G$ are labeled such that the condition (*) is satisfied. For any $i \in [h]:= \{1, \dots, h\}$, set
\[
N_i:=\{k \in [h] : x_ky_i \in E(G)\} \setminus \{i\},
\]
and define the base graph $O_i(G)$ as follows
\[
O_i(G):=G-\{x_ky_i : k\in N_i\} + \{x_kx_i : k \in N_i\}.
\]
Now we define the weight $w'$  on $O_i(G)$
by
\[
 w'(e)=
\left\{
\begin{array}{ll}
w(x_ky_i) &  \mbox{if} \ e=x_kx_i, \  k \in N_i\\
w(e)     & \mbox{otherwise.}
\end{array}
\right.
\]
Finally, we set$$O_i(G_w):=O_i(G)_{w'}.$$

We are now ready to prove the second main result of this paper.

\begin{thm} \label{CM}
Let $G$ be a Cohen-Macaulay very well-covered graph and let $w$ be an edge weight on $G$. Then the following conditions are equivalent.
\begin{enumerate}
\item $I(G_w)$ is an unmixed ideal.
\item $S/I(G_w)$ is a Cohen-Macaulay ring.
\end{enumerate}
\end{thm}

\begin{proof}

The implication (2) $\Longrightarrow$ (1) is well known. So, we prove (1) implies (2). As $G$ is a Cohen-Macaulay very well-covered graph, we may assume that conditions (*) and (**) are satisfied. In particular, $|V(G)|=2h$. It follows from unmixedness of $I(G_w)$ that the height of every associated prime of $S/I(G_w)$ is $h$. Thus, for every $\mathfrak{p}\in \Ass S/I(G_{w})$ and for every integer $k$ with $1\leq k\leq h$, exactly one of $x_k$ and $y_k$ belongs to $\mathfrak{p}$.

We use induction on $m:=\sum_{i=1}^{h}\deg_Gy_i$.
For $m=h$, the assertion follows from \cite[Theorem 5.7]{PaSW}. Hence, suppose $m>h$. Then there exists an integer $k$ with $1\leq k\leq h$ such that $\deg y_k\ge 2$. By contradiction, assume that $S/I(G_w)$ is not Cohen-Macaulay. Set $G'_{w'}=O_k(G_w)$. Using Theorem \ref{th:unmix}, one can easily check that $I(G'_{w'})$ is an unmixed ideal.
By induction hypotheses $S/I(G'_{w'})$ is Cohen-Macaulay. Therefore,$$(S/I(G_{w}))/(x_k-y_k)\cong (S/I(G'_{w'}))/(x_k-y_k)$$is Cohen-Macaulay.
Since $S/I(G_w)$ is not Cohen-Macaulay, $x_k-y_k$ is not regular on $S/I(G_{w})$. Hence,$$x_k-y_k\in \bigcup_{\mathfrak{p}\in \Ass S/I(G_{w})} \mathfrak{p}.$$
Thus, there exists an associated prime ideal $\mathfrak{p}$ of $S/I(G_{w})$ such that $x_k-y_k \in \mathfrak{p}$. Consequently, $x_k, y_k \in \mathfrak{p}$. This is a contradiction and proves that $S/I(G_w)$ is Cohen-Macaulay.
\end{proof}

It is well known (and easy to prove) that every unmixed bipartite graph is very well-covered. Hence, as an immediate consequence of Theorem \ref{CM}, we obtain the following corollary.

\begin{cor}
Let $G$ be a Cohen-Macaulay bipartite graph and let $w$ be an edge weight on $G$. Then the following conditions are equivalent.
\begin{enumerate}
\item $I(G_w)$ is an unmixed ideal.
\item $S/I(G_w)$ is a Cohen-Macaulay ring.
\end{enumerate}
\end{cor}


\section{Examples} \label{sec4}

Let $\mathcal{D}$ be a vertex-weighted oriented graph and let $G$ be its underlying graph. As we mentioned in Section \ref{sec1}, Pitones, Reyes and Toledo conjectured that $S/I(\mathcal{D})$ is Cohen-Macaulay, if $I(\mathcal{D})$ is unmixed and $S/I(G)$ is Cohen-Macaulay (see Conjecture \ref{ConjHLMRV}). The following example shows that the assertion of Conjecture \ref{ConjHLMRV} is not true.

\begin{exmp} \label{exmp1}
Let $\mathbb{K}$ be a field with ${\rm char}(\mathbb{K})=0$ and let $\mathcal{D}$ be the oriented graph with vertex set $V(\mathcal{D})=\{x_1,\ldots, x_{11}\}$ and edge set
\begin{equation*}
\begin{split}
E(\mathcal{D})=&\big\{(x_1,x_3), (x_1,x_4), (x_7,x_1), (x_1,x_{10}), (x_1,x_{11}), (x_2,x_4), (x_2,x_5),\\
        &\, \, \,  (x_2,x_8), (x_2,x_{10}), (x_2,x_{11}), (x_3,x_5), (x_3,x_6), (x_3,x_8), (x_3,x_{11}),\\
        &\, \, \,  (x_4,x_6), (x_4,x_9), (x_4,x_{11}), (x_7,x_5), (x_5,x_9), (x_{11},x_5), (x_6,x_8),\\
        &\, \, \,  (x_6,x_9), (x_9,x_7), (x_7,x_{10}), (x_8,x_{10})\big\}.
\end{split}
\end{equation*}
Consider the weight functions
\[
 w_1(x_i)=
\left\{
\begin{array}{ll}
1 &  \mbox{if} \ i\neq 11\\
2 &  \mbox{if} \ i=11,
\end{array}
\right.
\]
and
\[
 w_2(x_i)=
\left\{
\begin{array}{ll}
1 &  \mbox{if} \ i\neq 7\\
2 &  \mbox{if} \ i=7.
\end{array}
\right.
\]
For $i=1,2$, let $\mathcal{D}_i$ be the vertex-weighted oriented graph obtained from $\mathcal{D}$ by considering the weight function $w_i$. Then
\begin{equation*}
\begin{split}
I(\mathcal{D}_1)=&(x_1x_3, x_1x_4, x_1x_7, x_1x_{10}, x_1x_{11}^2, x_2x_4, x_2x_5, x_2x_8, x_2x_{10}, x_2x_{11}^2,  \\
        &\, \, \, x_3x_5, x_3x_6, x_3x_8, x_3x_{11}^2,  x_4x_6, x_4x_9, x_4x_{11}^2, x_5x_7, x_5x_9, x_5x_{11}, \\
        &\, \, \,  x_6x_8, x_6x_9, x_7x_9, x_7x_{10}, x_8x_{10}),
\end{split}
\end{equation*}
and
\begin{equation*}
\begin{split}
I(\mathcal{D}_2)=&(x_1x_3, x_1x_4, x_1x_7, x_1x_{10}, x_1x_{11}, x_2x_4, x_2x_5, x_2x_8, x_2x_{10}, x_2x_{11},  \\
        &\, \, \, x_3x_5, x_3x_6, x_3x_8, x_3x_{11},  x_4x_6, x_4x_9, x_4x_{11}, x_5x_7, x_5x_9, x_5x_{11}, \\
        &\, \, \,  x_6x_8, x_6x_9, x_7^2x_9, x_7x_{10}, x_8x_{10}).
\end{split}
\end{equation*}
Let $G$ be the underlying graph of $\mathcal{D}$. The edge ideal $I(G)$ of $G$ comes from the triangulation of the real projective plane (see for example
\cite[Exercise 6.3.65]{Vi1}). It is known that $S/I(G)$ is Cohen-Macaulay. However, for $i=1,2$, \textit{Macaulay2} computation shows that $I(\mathcal{D}_i)$ is unmixed but not Cohen-Macaulay, disproving Conjecture \ref{ConjHLMRV}.
We show that $S/I(\mathcal{D}_1)$ satisfies the Serre's condition $(S_2)$ condition, while $S/I(\mathcal{D}_2)$ does not.
Using \textit{Macaulay2} we know that $\depth  S/I(\mathcal{D}_i)=2$ for $i=1,2$. Since for $i=1,2$, $\dim  S/I(\mathcal{D}_i)=3$,
the quotient ring $S/I(\mathcal{D}_i)$ satisfies $(S_2)$ condition if and only if$$\dim \Ext_S ^{9}(S/I(\mathcal{D}_i), S) =
\dim \Ext_S ^{11-2}(S/I(\mathcal{D}_i), S) \le 2-2=0,$$by Lemma \ref{Serre}. With  \textit{Macaulay2}, one can check that
$\dim \Ext_S ^{9}(S/I(\mathcal{D}_1), S) =0$ and $\dim \Ext_S ^{9}(S/I(\mathcal{D}_2), S) =1$.
\end{exmp}

The following example provide counterexamples for the edge-weighted version of Conjecture \ref{ConjHLMRV}.

\begin{exmp}
Let $\mathbb{K}$ be a field with ${\rm char}(\mathbb{K})=0$ and let $G$ be the same graph as in Example \ref{exmp1}. Consider the following weighted edge ideals.
\begin{equation*}
\begin{split}
I(G_{w_1})=&(x_1x_3, x_1x_4, x_1x_7, x_1x_{10}, x_1x_{11}, x_2x_4, x_2x_5, x_2x_8, x_2x_{10}, x_2x_{11},  \\
        &\, \, \, x_3x_5, x_3x_6, x_3x_8, x_3x_{11},  x_4x_6, x_4x_9, x_4x_{11}, x_5x_7, x_5x_9, x_5x_{11}, \\
        &\, \, \,  x_6x_8, x_6x_9, x_7x_9, x_7x_{10}, x_8^2x_{10}^2).
\end{split}
\end{equation*}
\begin{equation*}
\begin{split}
I(G_{w_2})=&(x_1^2x_3^2, x_1^2x_4^2, x_1^2x_7^2, x_1^2x_{10}^2, x_1^2x_{11}^2, x_2^2x_4^2, x_2^2x_5^2, x_2^2x_8^2, x_2^2x_{10}^2, x_2^2x_{11}^2,  \\
        &\, \, \, x_3^2x_5^2, x_3^2x_6^2, x_3^2x_8^2, x_3^2x_{11}^2,  x_4^2x_6^2, x_4^2x_9^2, x_4^2x_{11}^2, x_5^2x_7^2, x_5^2x_9^2, x_5^2x_{11}^2, \\
        &\, \, \,  x_6^2x_8^2, x_6^2x_9^2, x_7^2x_9^2, x_7^2x_{10}^2, x_8x_{10}).
\end{split}
\end{equation*}
Then $S/I(G)$ is Cohen-Macaulay. However, \textit{Macaulay2} computation shows that $I(G_{w_1})$ is unmixed, but $S/I(G_{w_1})$ does not satisfy the Serre's condition $(S_2)$. On the other hand, \textit{Macaulay2} computation shows that $I(G_{w_2})$ is unmixed and $S/I(G_{w_2})$ satisfies the Serre's condition $(S_2)$ condition, but it is not Cohen-Macaulay.
\end{exmp}

\section*{Acknowledgment}
This work was partially supported by JSPS Grant-in Aid for Scientific Research (C) 18K03244.



\end{document}